\documentclass[12pt,reqno]{amsart}
\usepackage{amssymb}
\usepackage{amsmath}
\usepackage{amsfonts}
\usepackage[all]{xy}
\usepackage{amsthm}

\newcommand{\C}{\mathbb{C}}
\newcommand{\R}{\mathbb{R}}
\newcommand{\g}{\frak{g}}

\newtheorem{dfn}{Definition}[section]
\newtheorem{thm}[dfn]{Theorem}
\newtheorem{prop}[dfn]{Proposition}

\begin{document}

\title[Sasakian geometry and Heisenberg groups]{Sasakian geometry and Heisenberg groups}

\author[I. Biswas]{Indranil Biswas}

\address{Department of Mathematics, Shiv Nadar University, NH91, Tehsil Dadri,
Greater Noida, Uttar Pradesh 201314, India}

\email{indranil.biswas@snu.edu.in, indranil29@gmail.com}

\author[H. Kasuya]{Hisashi Kasuya}

\address{Department of Mathematics, Graduate School of Science, Osaka University, Osaka, Japan}

\email{kasuya@math.sci.osaka-u.ac.jp}

\thanks{This work was partially supported by JSPS KAKENHI Grant Number 19H01787.}



\begin{abstract}
We prove that a compact Sasakian manifolds whose first and second basic Chern classes vanish is locally 
isomorphic to the real Heisenberg group equipped with the standard left invariant Sasakian structure up to 
deformation associated to a basic $1$-from.
\end{abstract}

\maketitle

\section{Introduction}

The $(2n+1)$-dimensional real Heisenberg group $H_{2n+1}$ is the group of $(n+2)\times (n+2)$
matrices of the form
\[\left(
\begin{array}{ccc}
1& x&t \\
0& I &\,^{t} y\\
0&0&1
\end{array}
\right)
\]
where $I$ is the $n\times n$ unit matrix, $x,\,y \,\in\, \R^{n}$ (which is identified with $1\times n$ real 
matrices) and $t\,\in\, \R$. It is well-known that the standard contact structure $\eta_{st}$ on 
$\R^{2n+1}$ can be seen as a left-invariant contact structure on $H_{2n+1}$. We can define a natural 
left-invariant CR structure $T^{1,0}_{H_{2n+1}}$ on $H_{2n+1}$ so that the pair $(T^{1,0}_{H_{2n+1}}, \,
\eta_{st})$ is a Sasakian structure on $H_{2n+1}$ (see Section
\ref{se2.2}. We call $(T^{1,0}_{H_{2n+1}},\, \eta_{st})$ the standard 
left-invariant Sasakian on $H_{2n+1}$ (see \cite{BG} for Sasakian manifolds).

The purpose of this paper is to characterize compact Sasakian manifolds which are uniformized by the 
Sasakian manifold $(H_{2n+1},\, T^{1,0}_{H_{2n+1}},\, \eta_{st})$. For this, we use basic Chern classes 
$c_{i,B}(T^{1,0}_{M})$ associated to a compact Sasakian manifold $(M,\, T_{M}^{1,0}, \,\eta)$. We prove 
the following:

\begin{thm}\label{Heisen}
Let $(M,\, T_{M}^{1,0}, \,\eta)$ be a $(2n+1)$-dimensional compact Sasakian manifold.
If $c_{1, B}(T^{1,0}_{M})\,=\,0$ and $c_{2, B}(T^{1,0}_{M})\,=\,0$, then there is an $A_{B}^1$-deformation
of $(T_{M}^{1,0}, \eta)$ such that the universal covering of $M$ with the Sasakian structure induced by
the deformed Sasakian structure is isomorphic to the $(2n+1)$-dimensional real Heisenberg group with the
standard left-invariant Sasakian structure.
\end{thm}

The above mentioned $A_{B}^1$-deformations are deformations of the Sasakian structure
$(T_{M}^{1,0}, \,\eta)$ on $M$ defined by a basic $1$-form on $M$ without changing the Reeb vector field.
This type of deformations are essential for classification of compact Sasakian manifolds (see \cite{Bel2, KM}).
Theorem \ref{Heisen} is a Sasakian version of the statement that compact K\"ahler manifold with
vanishing first and second Chern classes is flat; this result on compact K\"ahler manifolds is
proved by Yau's theorem \cite{Y} proving Calabi's conjecture (see \cite[Corollary 4.15]{Ko}). 
By the Bieberbach theorem, a compact flat K\"ahler manifold is a finite quotient of a complex torus.

A compact quotient $\Gamma\backslash H_{2n+1}$ of $H_{2n+1}$ by a discrete subgroup $\Gamma $ is called a 
Heisenberg nilmanifold. As a flat manifold, by the generalized Bieberbach theorem (see \cite{DLR}), we can 
say that if the universal covering of a compact Sasakian manifold is the $(2n+1)$-dimensional real 
Heisenberg group with the standard left-invariant Sasakian structure, then it is a finite quotient of a 
Heisenberg nilmanifold.

The special case of Theorem \ref{Heisen} where $n\,=\,1$ is a consequence of Belgun's classification of 
three dimensional compact Sasakian manifolds in \cite{ Bel1, Bel2} proved by the Enriques-Kodaira 
Classification. Another part of Belgun's classification is also extended to higher dimensions in \cite{KM}.

\bigskip

\noindent {\bf Acknowledgement.\,} IB wishes to thank Osaka University for hospitality while the
work was carried out.

\section{Sasakian Geometry}

\subsection{Sasakian structures}

Let $M$ be a $(2n+1)$-dimensional real smooth manifold. A {\em CR-structure} on $M$ is an $n$-dimensional 
complex involutive sub-bundle $T^{1,0}_{M}$ of the complexified tangent bundle $TM_{\C}\,=\, 
TM\otimes_{\mathbb R} {\C}$ such that $T^{1,0}_{M}\cap T^{0,1}_{M}\,=\,\{0\}$, where 
$T^{0,1}_{M}\,=\,\overline{T^{1,0}_{M}}$. Define the real $2n$-dimensional sub-bundle
\begin{equation}\label{s1}
S\,:=\,TM\cap (T^{1,0}_{M}\oplus T^{0,1}_{M})\subset TM.
\end{equation}
We have the almost complex structure $I\,:\, S\,\longrightarrow
\, S$ associated to the decomposition $ S_{\C}\,=\,T^{1,0}_{M}\oplus T^{0,1}_{M}$.

A {\em 
strongly pseudo-convex CR structure} on $M$ is a pair $(T^{1,0}_{M} ,\, \eta) $ consisting
of a CR structure $T^{1,0}_{M}$ and a contact 
$1$-form $\eta$ such that $\ker\eta\,=\,S$ (see \eqref{s1}) and the bilinear form on $S$ defined by 
$L_{\eta}(X,\,Y)\,=\,d\eta(X, IY)$ is a Hermitian metric on $(S,\, I)$.
Take the Reeb vector field $\xi$ associated to the contact $1$-form $\eta$. Then
$(T^{1,0}_{M} ,\, \eta) $ is {\em Sasakian} if for any smooth section $X$ of $T^{1,0}_{M}$ the Lie bracket 
$[\xi,\, X]$ is also a section of $T^{1,0}_{M}$.
We consider the $1$-dimensional foliation ${\mathcal F}_{\xi}$ on $M$ generated by $\xi$.
If $(T^{1,0}_{M} ,\, \eta) $ is Sasakian, then $T^{1,0}_{M}$ defines a transverse holomorphic structure
on ${\mathcal F}_{\xi}$ and $d\eta$ is a transverse K\"ahler form on ${\mathcal F}_{\xi}$.

Let $(M,\, T^{1,0}_{M} ,\, \eta) $ be a Sasakian manifold.
A differential form $\omega$ on $M$ is called {\em basic} if the equations
\[
i_{\xi}\omega\,=\,0\,=\, i_{\xi} d\omega
\]
hold. We denote by $A^{\ast}_{B}(M)$ the subspace of basic 
forms in the de Rham complex $A^{\ast}(M)$. Then
$A^{\ast}_{B}(M)$ is a sub-complex of the de Rham complex $A^{\ast}(M)$. Denote by $H_{B}^{\ast}(M)$
the cohomology of the basic de Rham complex $A^{\ast}_{B}(M)$. We note that $d\eta\,\in\,
A^{2}_{B}(M)$ and $[d\eta]\,\not=\,0\,\in\, H_{B}^{2}(M)$ if $M$ is compact.
We have the bigrading $A^{r}_{B}(M)_{\C}\,=\,\bigoplus_{p+q=r} A^{p,q}_{B}(M)$ as well as the decomposition
of the exterior differential $$d\big\vert_{A^{r}_{B}(M)_{\C}}\,=\,\partial_{B}+\overline\partial_{B}$$
on $A^{r}_{B}(M)_{\C}$, so that $$\partial_{B}\,:\,A^{p,q}_{B}(M)\,\longrightarrow\,
A^{p+1,q}_{B}(M)\ \ \text{ and }\ \
\overline\partial_{B}\,:\,A^{p,q}_{B}(M)\,\longrightarrow\,
A^{p,q+1}_{B}(M)\, .$$ 
We note that $d\eta\,\in\, A^{1,1}_{B}(M)$.
For a strongly pseudo-convex CR-manifold $(M, \,T^{1,0}_{M} ,\, \eta)$
there exists a unique affine connection $\nabla^{TW}$ on $TM$ such that the following
conditions hold (\cite{Tan, Web}):
\begin{enumerate}
\item $S$ is parallel with respect to $\nabla^{TW}$,

\item $\nabla^{TW}I\,=\,\nabla^{TW}d\eta\,=\,\nabla^{TW}\eta\,=\,\nabla^{TW}\xi\,=\,0$, and

\item the torsion $T^{TW}$ of the affine connection $\nabla^{TW}$ satisfies the equation
\[T^{TW} (X,\,Y)\,=\, -d\eta (X,\,Y)\xi
\]
for all $X,\,Y\,\in\, S_{x}$ and $x\,\in\, M$.
\end{enumerate}
This affine connection $\nabla^{TW}$ is called the {\em Tanaka--Webster connection}. It
is known that $(T^{1,0},\,\eta)$ is a Sasakian manifold if and only if 
$T^{TW}(\xi,\,v)\,=\,0$ for all $v\, \in\, TM$. 

Consider the unique homomorphism $\Phi_{\xi}\,:\, TM\,\longrightarrow\, TM $ which extends
$I\,:\, S\,\longrightarrow\, S$ satisfies the condition $\Phi_{\xi}(\xi)\,=\,0$.
Then, $\Phi^{2}_{\xi}\,=\,-{\rm Id}+\xi\otimes \eta$.
We call $\Phi_{\xi}$ the {\em almost contact structure} associated
to a Sasakian structure $(T^{1,0}_{M} ,\, \eta) $.
We define the Riemannian metric $$g_{\eta}(X,\,Y)\,=\,\eta(X)\eta(Y)+ d\eta(X,\, \Phi_{\xi} Y).$$
We call $g_{\eta}$ the {\em Sasakian metric} associated to the Sasakian structure $(T^{1,0}_{M} ,\, \eta) $.

\subsection{Heisenberg groups}\label{se2.2}

We consider the $(2n+1)$-dimensional real Heisenberg group $H_{2n+1}$ which is the group of matrices of the 
form
\[\left(
\begin{array}{ccc}
1& x&t \\
0& I &\,^{t} y\\
0&0&1
\end{array}
\right)
\]
where $I$ is the $n\times n$ unit matrix, $x,\,y \,\in\, \R^{n}$ and $t\,\in\, \R$.
Denote by ${\frak h}_{2n+1}$ the Lie algebra of $H_{2n+1}$.
${\frak h}_{2n+1}$ has the basis
\[
X_{i}\,=\,\frac{\partial}{\partial x_{i}},\ \, Y_{i}\,=\,
\frac{\partial}{\partial y_{j}}+x_{i}\frac{\partial}{\partial t},\ \, T\,=\,\frac{\partial}{\partial t}
\]
such that $[X_{i},\,Y_{i}]\,=\,-T $ for every $i$ and all other brackets are $0$.
We have the dual basis 
\[\alpha_{i}\,=\,dx_{i},\ \, \beta_{i}\,=\,dy_{i},\ \, \gamma\,=\,dt-\sum x_{i}dy_{i}.
\]
Define the CR structure $T^{1,0}_{H_{2n+1}}$ generated by $ X_{1}-\sqrt{-1}Y_{1},\, \cdots ,
\, X_{n}-\sqrt{-1}Y_{n}$ and $\eta_{st}\,=\,-\gamma$.
Then the pair $(T^{1,0}_{H_{2n+1}},\, \eta_{st})$ is a left-invariant Sasakian structure on $H_{2n+1}$.
We call this Sasakian structure the standard left-invariant Sasakian structure on $H_{2n+1}$.
The Tanaka-Webster connection $\nabla^{TW}$ is the left-invariant affine connection such that
$X_{1},\,\cdots,\, X_{n},\, Y_{1},\,\cdots ,\,Y_{n},\, T$ are parallel.

A nilmanifold is a compact quotient $\Gamma\backslash G$ of a simply connected nilpotent Lie group $G$ by a 
discrete subgroup $\Gamma$. If $G\,=\,H_{2n+1}$, then we call it a Heisenberg nilmanifold. Any Heisenberg 
nilmanifold admits a Sasakian structure which is induced by the standard left-invariant Sasakian structure 
on $H_{2n+1}$. Conversely, if a nilmanifold $\Gamma\backslash G$ admits a Sasakian structure, then 
$G\,=\,H_{2n+1}$ and hence $\Gamma\backslash G$ is diffeomorphic to a Heisenberg nilmanifold (see 
\cite{CDMY,Ka17}).

\subsection{Deformations}

Let $(M,\, T^{1,0}_{M},\,\eta)$ be a Sasakian manifold. A {\em $A^1_{B}$-deformation} of $( T^{1,0}_{M}, 
\,\eta)$ is a Sasakian structure $(T^{1,0}_{M\tau},\,\eta^{\tau})$ such that $\eta^{\tau}\,=\,\eta+\tau$ for 
$\tau\,\in\, A^1_{B}$ and $$T^{1,0}_{M\tau}\,=\,\{X-\tau(X)\xi\,\,\big\vert\ \, X\,\in\, T^{1,0}_{M}\}.$$ 
The Reeb vector field of $\eta^{\tau}$ is $\xi$, and the transverse holomorphic structure 
$T^{1,0}_{M\tau}\oplus \langle \xi\rangle$ induced by $T^{1,0}_{M\tau}$ is same as the one induced by 
$T^{1,0}_{M}$. If another contact structure $\eta^{\prime}$ has the Reeb vector field $\xi$, and 
$d\eta^{\prime}$ is transverse K\"ahler for $T^{1,0}_{M}\oplus \langle \xi\rangle$, then we have a 
$A^1_{B}$-deformation $( T^{1,0}_{M\tau},\,\eta^{\tau})$ for $\tau\,=\,\eta-\eta^{\prime}$.

\section{Basic vector bundles over Sasakian manifolds}\subsection{Basic vector bundles}

We define structures on smooth vector bundles over Sasakian manifolds $(M, \,T_{M}^{1,0},\, \eta)$ related to 
${\mathcal F}_{\xi}$ in terms of Rawnsley's partial flat connections in \cite{Ra}.

A structure of basic vector bundle on a $C^\infty$ vector bundle $E$ is a 
linear operator $\nabla_{\xi}\,:\, {\mathcal C}^{\infty}(E)
\,\longrightarrow\, {\mathcal C}^{\infty}(E)$ such that
\[ \nabla_{\xi}(f s)\,=\,f\nabla_{\xi} s+ \xi(f) s.
\]

A differential form $\omega\,\in\, A^{\ast}(M,\,E)$ with 
values in $E$ is called basic if 
the equations
\[
i_{\xi}\omega\,=\,0\,=\,\nabla_{\xi} \omega
\]
hold, where we extend $\nabla_{\xi}\,:\, A^{\ast}(M,\,E)\,\longrightarrow\, A^{\ast}(M,\,E)$. Let 
$A^{\ast}_{B}(M,\, E)\,\subset\,A^{\ast}(M,\,E)$ be the subspace of basic forms in the space 
$A^{\ast}(M,\,E)$ of differential forms with values in $E$. A Hermitian metric on $E$ is called {\em basic} 
if it is $\nabla_{\xi}$-invariant. Unlike the usual Hermitian metric, a basic vector bundle may not admits a 
basic Hermitian metric in general (see \cite{BK2}). For a connection operator $\nabla \,:\, A^{\ast}(M,\,E)\, 
\longrightarrow\, A^{\ast+1}(M,\,E)$ such that the covariant derivative of $\nabla$ along $\xi$ is 
$\nabla_{\xi}$, the curvature $R^{\nabla}\,\in\, A^{2}(M,\, {\rm End} (E))$ is actually
basic, meaning $R^{\nabla} 
\,\in\, A^{2}_{B}(M,\, {\rm End} (E))$. For the Chern forms $c_{i}(E,\, \nabla)\,\in\, A^{2i}(M)$ associated 
with $\nabla$, we have $c_{i}(E,\,\nabla)\,\in\, A_{B}^{2i}(M)$ and we define the basic Chern classes 
$c_{i,B}(E)\,\in\, H^{2i}_{B}(M)$ by the cohomology classes of $c_{i}(E,\,\nabla)$.

A structure of basic holomorphic bundle on a $C^\infty$ vector bundle $E$ is a linear
differential operator $\nabla^{\prime\prime}\,:\, {\mathcal C}^{\infty}(E)
\,\longrightarrow\,
{\mathcal C}^{\infty}(E\otimes ( \langle \xi\rangle \oplus T^{0,1}_{M})^{\ast})$ such that
\begin{itemize}
\item for any $X\,\in\, \langle \xi\rangle \oplus T^{0,1}_{M}$, and any smooth function $f$ on $M$, the equation
\[ \nabla^{\prime\prime}_{X}(f s)\,=\,f\nabla^{\prime\prime}_{X} s+ X(f) s
\]
holds for all smooth sections $s$ of $E$, and

\item if we extend $\nabla^{\prime\prime}$ to $$\nabla^{\prime\prime}\,:\, {\mathcal C}^{\infty}(E\otimes 
\bigwedge^{k} (\langle \xi\rangle \oplus T^{0,1}_{M})^{\ast}) \, \longrightarrow\, {\mathcal 
C}^{\infty}(E\otimes \bigwedge^{k+1} (\langle \xi\rangle \oplus T^{0,1}_{M})^{\ast}),$$ then 
$\nabla^{\prime\prime}\circ \nabla^{\prime\prime}\,=\,0$.
\end{itemize}

A basic holomorphic vector bundle $E$ has a canonical basic bundle structure corresponding to the derivative
$\nabla^{\prime\prime}_{\xi}$. Then $\nabla^{\prime\prime}$ defines the linear operator $\overline{\partial}_{E}
\, : \, A^{p,q}_{B}(M,\, E)\, \longrightarrow\, A^{p,q+1}_{B}(M,\, E)$ so that 
$$\overline{\partial}_{E}( f\omega )\,=\,\overline{\partial}_{B}f \wedge \omega +
f \overline{\partial}_{E}( \omega )$$ for $f\,\in\, A^{0}_{B}(M)$ and $\omega\,\in\, A^{p,q}_{B}(M,\, E)$.
If a basic holomorphic vector bundle $E$ admits a basic Hermitian metric $h$, as complex case, we have a
unique unitary connection $\nabla^{h}$ such that
$\nabla^{h}_{X}\,=\,\nabla^{\prime\prime}_{X}$ for all $X\,\in\, \langle \xi\rangle \oplus T^{0,1}_{M}$.
The Tanaka-Webster connection $\nabla^{TW} $ defines a 
structure of a basic holomorphic bundle on $T^{1,0}_{M}$, and $\nabla^{TW}_{\vert T^{1,0}_{M}}
\,=\,\nabla^{h}$ for the basic Hermitian metric $h$ defined by the transverse K\"ahler structure $d\eta$.

The basic Hermitian metric $h$ is said to be
{\em Hermite-Einstein } if the equation
\begin{equation}\label{hyme}
\Lambda R^{\nabla^{h} }\,\, =\,\, \lambda I
\end{equation}
holds for some constant $\lambda$.

The following is obtained the same way as in the K\"ahler case (\cite{Ko, SimC, BK, BK2}):
\begin{thm}\label{BGeq}
Suppose a basic holomorphic vector bundle $E$ admits a Hermite-Einstein metric $h$.
Then the inequality
\begin{equation}\label{c1}
\int_{M} \left(2c_{2, B}(E) -
\frac{r-1}{r}c_{1, B}(E)^2\right)\wedge(d\eta)^{n-2}\wedge\eta \, \geq\, 0
\end{equation}
holds. Furthermore, if the equality holds, then $R^{\nabla^{h} }\, =\, \lambda I$.
If $c_{1, B}(E)\,=\,0$ and $c_{2, B}(E)\,=\,0$, then $R^{\nabla^{h} }\,=\,0$ and hence $E$ is unitary flat.
\end{thm}

\section{Proof of Theorem \ref{Heisen}}

Let $(M, \,T_{M}^{1,0},\, \eta)$ be a $(2n+1)$-dimensional compact Sasakian manifold.
Suppose $c_{1, B}(T^{1,0}_{M})\,=\,0$ and $c_{2, B}(T^{1,0}_{M})\,=\,0$.

\begin{prop}
There exists a $A^1_{B}$--deformation $(T^{1,0}_{M\tau},\,\eta^{\tau})$ of $(T_{M}^{1,0},\, \eta)$ such that 
$ T_{M}^{1,0}$ is unitary flat with respect to a Basic metric induced by $\eta^{\tau}$.
\end{prop}

\begin{proof}
Since $c_{1, B}(T^{1,0}_{M})\,=\,0$, Yau's theorem on a transverse K\"ahler structure \cite[3.5.5]{EKA} 
implies the existence of a basic Ricci flat K\"ahler metric $\omega$ in the basic cohomology class 
$[d\eta]_{B}$. In view of the conditions $c_{1, B}(T^{1,0}_{M})\,=\,0$ and $c_{2, B}(T^{1,0}_{M})\,=\,0$,
considering this Ricci flat metric $\omega$ as an 
Hermite-Einstein metric $h$ on $T^{1,0}$, we conclude by Theorem \ref{BGeq} that $T^{1,0}$ is unitary flat. Since 
$[d\eta-\omega]_{B}\,=\,0$, we have a basic $1$-form $\tau\,\in\, A^1_{B}$ such that $d\eta-\omega 
\,=],d\tau$. Hence the $A^1_{B}$--deformation $(T^{1,0}_{M\tau},\,\eta^{\tau})$ is desired one.
\end{proof}

We assume that $d\eta$ induces a flat Hermitian metric on $T^{1,0}_{M}$. It suffices to prove that the 
universal covering of the compact Sasakian manifold $(M,\, T_{M}^{1,0}, \,\eta)$ is isomorphic to the 
$(2n+1)$-dimensional real Heisenberg group $H_{2n+1}$ with the standard left-invariant Sasakian structure 
$(T_{H_{2n+1}}^{1,0},\,\eta_{st})$. Since the Tanaka-Webster connection $\nabla^{TW}$ is a unitary flat
connection on $T^{1,0}_{M}$, it follows that
$\nabla^{TW}$ is a Riemannian flat connection on $TM$. Since 
$\nabla^{TW}d\eta\,=\,\nabla^{TW}\xi\,=\,0$, the torsion $T^{TW}$ is parallel with respect to $\nabla^{TW}$. As 
the $\nabla^{TW}$ is complete (see \cite{Fri}), we can apply \cite[Structure Theorem 3.3]{KT}.

Consider the universal covering $\widetilde{M}$ with the lifted Sasakian metric $(T_{\widetilde{M}}^{1,0}, 
\,\eta)$. Define the space $\g$ of the vector fields on $\widetilde{M}$ which are invariant under the 
parallel transport for $\nabla^{TW}$. Then $\g_{x}\,=\, T_{x}\widetilde{M}$ for each $x\,\in\, 
\widetilde{M}$, and $\g$ is a Lie subalgebra of the Lie algebra of vector fields on $\widetilde{M}$ such 
that $[X,\,Y]\,=\,-T^{TW}(X,Y)$ for all $X,\,Y\,\in \,\g$. From $\nabla^{TW}\xi\,=\,0$ it follows that 
$\xi\,\in \,\g$. Since $S$ in \eqref{s1} is parallel with respect to $\nabla^{TW}$, we can have a subspace 
$\frak{s}\,\subset\, \g$ such that $\frak{s}_{x}\,=\,S_{x}$ for each $x\,\in\, \widetilde{M}$. Thus, for 
$X,\,Y\,\in\, \frak{s}$, we have $[X,\,Y]\,=\,d\eta(X,\,Y)\xi $.

Since $\g_{x}\,=\, T_{x}\widetilde{M}$ for each $x\,\in\, \widetilde{M}$, using the properties of 
$(T_{\widetilde{M}}^{1,0}, \eta)$, we can take, at a point $x_{0}\,\in\, \widetilde{M}$, a basis 
$\widetilde{X}_{1},\,\cdots,\, \widetilde{X}_{n},\, \widetilde{Y}_{1},\,\cdots,\, \widetilde{Y}_{n},\, 
\widetilde{T}$ of $\g$ such that $\widetilde{X}_{1x_{0}}-\sqrt{-1}\widetilde{Y}_{1x_{0}},\,\cdots,\, 
\widetilde{X}_{nx_{0}}- \sqrt{-1}\widetilde{Y}_{nx_{0}}$ is a basis of $T_{\widetilde{M}x_{0}}^{1,0}$, 
$T\,=\,-\xi$ and $d\eta(X_{i},\, Y_{j})\,=\,-\delta_{ij}$, $d\eta(X_{i},\, X_{j})\,=\,d\eta(Y_{i},\, 
Y_{j})\,=\,0$. Thus, we have an isomorphism $\g\,\cong\, {\frak h}_{2n+1}$. By \cite[Structure Theorem 3.3]{KT}, 
we have a connection preserving diffeomorphism $\widetilde{M}\,\cong\, H_{2n+1}$. Under this diffeomorphism, $\g$ 
corresponds to the Lie algebra of the left-invariant vector fields on $H_{2n+1}$. Since $\nabla^{TW}I=0$, the 
subspace $\frak{a}\,\subset\, \g_{\C}$ generated by $\widetilde{X}_{1}-\sqrt{-1}\widetilde{Y}_{1},\,
\cdots,\, \widetilde{X}_{n}- 
\sqrt{-1}\widetilde{Y}_{n}$ satisfies the condition $\frak{a}_{x}\,=\,T_{\widetilde{M}x}^{1,0}$ for
each $x\,\in\, \widetilde{M}$. Since 
$\nabla^{TW}\eta\,=\,0$, we have $\eta\,\in\, \g^{\ast}$ such that $\eta(\xi)\,=\,1$ and
${\rm kernel}(\eta)\,=\, \frak{s}$. Hence, the 
diffeomorphism $\widetilde{M}\,\cong\, H_{2n+1}$ is an isomorphism between the Sasakian manifolds
$(\widetilde{M}, \, T_{\widetilde{M}}^{1,0},\, \eta)$ and $(H_{2n+1},\, T_{H_{2n+1}}^{1,0},\, \eta_{st})$.
This completes the proof.

\end{document}